\pdfoutput=1
\RequirePackage{ifpdf}
\ifpdf 
\documentclass[pdftex]{sigma}
\else
\documentclass{sigma}
\fi

\DeclareMathOperator{\Sp}{Sp}
\DeclareMathOperator{\SL}{SL}
\DeclareMathOperator{\GL}{GL}
\DeclareMathOperator{\M}{M}
\newcommand{\deter}{\textnormal{det}_*}
\DeclareMathOperator{\n}{N}

\begin{document}

\allowdisplaybreaks

\newcommand{\arXivNumber}{1506.08071}

\renewcommand{\PaperNumber}{076}

\FirstPageHeading

\ShortArticleName{Weil Representation of a Generalized Linear Group}

\ArticleName{Weil Representation of a Generalized Linear Group\\ over a  Ring of Truncated Polynomials over a Finite\\ Field Endowed with a Second Class Involution}

\Author{Luis GUTI\'ERREZ FREZ~$^\dag$ and Jos\'e PANTOJA~$^\ddag$}

\AuthorNameForHeading{L.~Guti\'errez Frez and J.~Pantoja}

\Address{$^\dag$~Instituto de Ciencias F\'{\i}sicas y Matem\'aticas, Universidad Austral de Chile,\\
\hphantom{$^\dag$}~Campus Isla Teja SN, Edificio Pug\'in, Valdivia, Chile}
\EmailD{\href{mailto:luis.gutierrez@uach.cl}{luis.gutierrez@uach.cl}}

\Address{$^\ddag$~Instituto de Matem\'aticas, Pontificia Universidad Catolica de Valpara\'{i}so,\\
\hphantom{$^\ddag$}~Blanco Viel 596, Co.~Bar\'{o}n, Valpara\'{i}so, Chile}
\EmailD{\href{mailto:jpantoja@ucv.cl}{jpantoja@ucv.cl}}

\ArticleDates{Received July 03, 2015, in f\/inal form September 14, 2015; Published online September 26, 2015}

\Abstract{  We construct a complex linear Weil representation~$\rho$ of the generalized special linear group
$G=\SL_*^{1}(2,A_n)$ ($A_n=K[x]/\langle x^n\rangle $, $K$ the quadratic extension of the f\/inite f\/ield~$k$ of~$q$ elements, $q$~odd), where~$A_n$ is endowed with a second class involution. After the construction of a specif\/ic data, the representation is def\/ined on the generators of a~Bruhat presentation of~$G$, via linear operators satisfying the relations of the presentation. The structure of a~unitary group~$U$ associated to~$G$ is described. Using this group we obtain a~f\/irst decomposition of~$\rho$.}


\Classification{20C33; 20C15; 20F05}

\section{Introduction}

Weil representations is a central topic in representation theory. They arise as a consequence of one of the many seminal works of A.~Weil~\cite{weil}. They are projective representations of the groups $\Sp(2n,F)$, $F$ a locally compact f\/ield. These representations are an important subject of study. By decomposition into irreducible factors, Weil  representations provide all irreducible complex representations of the general linear group over a f\/inite f\/ield, and over a local f\/ield of residual characteristic dif\/ferent from~$2$. It has many other important consequences, and it has applications in dif\/ferent  topics as theta functions and physics to mention only some of them. In fact, in  representation theory  it helps to understand (among other things) the harmonic analysis of the group, and in the Langlands program, to explain the relations between linear groups def\/ined over local or global f\/ields, and the Galois groups of the f\/ields.

A point of view that has been favorably used in representation theory, is  to extend  to higher rank groups, methods successfully used in lower rank groups. This philosophy has led  Pantoja and Soto-Andrade   (see~\cite{P,ps,P-SA}) to def\/ine the groups $\GL_{\ast}^{\varepsilon}(2,A)$ and $\SL_{\ast}^{\varepsilon}(2,A)$, where \mbox{$\varepsilon=\pm 1$}. These are ``generalized linear groups'' of rank $2$ with coef\/f\/icients in a unitary involutive ring~$(A,*)$~\cite{book} ($\varepsilon=\pm1$). In this way, the symplectic group and the orthogonal group  of rank~$n$ over a~f\/ield~$F$ appear as  groups $\SL_{\ast}^{\varepsilon}(2,A)$ considering as involutive ring, the ring of $n\times n$ mat\-rices over~$F$ with the transposition of matrices, and taking, respectively, $\varepsilon$ equals to~$-1$ and~$1$. Dif\/ferent  choices of involutive rings produce new  examples of diverse  kind.

The  rings considered are, in general, non-commutative, and this non-commutativity is controlled by  the relation $(ab)^{\ast}=b^{\ast}a^{\ast}$. In this sense, it should be pointed out  the similarity of our group $\GL_{\ast}(2,A)$, with the quantum group $\GL^q(2)$ (as a ``group" of matrices with some com\-mu\-ting relations). Furthermore, the generalized linear groups af\/ford the notion of a $\ast$-determinant, $\deter$, in analogy with the $q$-determinant. Also, the homomorphism   $\ast$-determinant can be considered as a f\/irst way of  describing the mentioned groups, as these can be def\/ined as two by two matrices with coef\/f\/icients in~$A$, (satisfying some commutation relations that involve~$\ast$), with $\ast$-determinant dif\/ferent from~$0$. Then, the groups $\SL_{\ast}(2,A)$ appear as the kernel of the homomorphism~$\deter$.

The approach, used here, has been considered in  diverse interesting groups~\cite{Gu,G-P-SA,andr}.
Having a presentation of the group, that the authors call Bruhat presentation, and certain specif\/ic data, a generalized Weil representation has been constructed. Historically, the groups $G$ were def\/ined f\/irst for $\varepsilon=-1$. The work of Soto-Andrade on the symplectic group over a f\/inite f\/ield~\cite{SA} and the representations constructed by Guti\'{e}rrez~\cite{Gu} appear as  examples of these constructions in \cite{G-P-SA}. Later, after the generalization to $\varepsilon=1$,  Vera~\cite{andr} constructed a generalized Weil representation of the orthogonal group over a f\/inite f\/ield, where she also relates the representation with the one that can be constructed by the theory of dual pairs of Howe~\cite{H}. However, even that  the orthogonal group is a natural example of a generalized linear group with $\varepsilon=1$, her work is done treating the group as a ``symplectic type'' group, i.e., as a $\SL_{\ast}^{-1}$ group, by twisting the transposition of matrices to def\/ine the involution of the ring.

Our method to construct Weil representations is one of many successful approaches to attack this topic, and has been studied by several authors (see, e.g., \cite{A-M,Ge,Go,weil}). Using Weil's original point of view,  Szechtman et al.\ in~\cite{C-Sze} construct Weil representations of symplectic groups over f\/inite rings via Heisenberg groups. Recently, Herman and Szechtman~\cite{He-Sze} construct Weil representations of unitary groups associated to a~f\/inite, commutative, local principal ring of odd characteristic, by imbedding the group into a symplectic group.

A great diversity of groups can be originated via generalized linear groups for appropriate choices of involutive rings. For each of them, Weil representations could be constructed by  using our approach of def\/ining linear operators for each generator of the group in such a way that they satisfy the basic (universal) relations of a ``simple'' presentation. This variety of cases for which a (generalized) Weil representation could be produced has led  us to verify that, in practice, the procedure is ef\/fective. In fact, several examples of groups $\GL_*^{\varepsilon}(2,A)$ and $\SL_*^{\varepsilon}(2,A)$ have been given in loc.\ cit.\ for dif\/ferent choices of involutive rings provided with  f\/irst class involutions~\cite{book}. However, no example with an involutive ring (algebra) provided with a second class involution has been considered up to now. Furthermore, explicit constructions for generalized linear groups with~$\varepsilon=1$ have not been made so far.

\looseness=-1
In this work, we construct a Weil representation of a generalized $\SL_*^1$ ``orthogonal type group'' (i.e., a generalized special linear group with $\varepsilon=1$) over the ring $A_n=K[x]/\langle x^n\rangle $, which is a~non-semisimple ring (algebra in fact) over~$K$, $K$ the quadratic extension of the f\/inite f\/ield~$k$ of~$q$ elements ($q$~odd), provided with a second class involution. We show f\/irst that the group under consideration has a Bruhat presentation, after which a data necessary to produce a Weil representation of $G$ via relations and generators, is specif\/ically described. We also obtain a f\/irst decomposition of the representation.
Using a complete dif\/ferent approach leaning on the works of Amritanshu Prasad and Kunal Dutta~\cite{Pr-Du,Pr}, a further decomposition could be achieved. Furthermore, a comparison of their methods and ours  will be performed in a work that will appear elsewhere.

The paper is organized as follows: In Section~\ref{section2}, we present the main def\/initions on generalized classical groups $\GL_*^{\varepsilon}(2,A)$ and
$\SL_*^{\varepsilon}(2,A)$ for an involutive ring $(A,*)$ and we describe some properties of the truncated polynomials ring
$K[x]/\langle x^n\rangle$, for~$K$  a quadratic extension of a f\/ield of~$q$ elements ($q$ odd). In Section~\ref{section3},  a Bruhat presentation of $\SL_*^{1}(2,A_n)$ is constructed. Section~\ref{section4} is devoted f\/irst to recall a very general procedure to construct generalized Weil representations of $\SL_*^{\varepsilon}(2,A)$, for a group with a Bruhat presentation and a suitable data $(M,\chi,\gamma,c)$. After this, the necessary data for the group under consideration is produced and completely described and detailed.
Finally, in Section~\ref{section5}, we def\/ine the abelian ``unitary'' group $U(M,\chi,\gamma,c)$ of~$(M,\chi,\gamma,c)$, with the explicit decomposition into cyclic subgroups, which allow us to get a~f\/irst decomposition of the constructed representation.

\section{Preliminaries}\label{section2}

In this section, we f\/ix notations and recall some basic facts about generalized
general special linear groups over involutive rings.

Let $A$ be a unitary ring endowed with an involution $\ast$, i.e., an
antiautomorphism $a\mapsto a^{\ast}$  of~$A$ of order two. We denote by~$Z(A)$
the center of $A$, and we write $A^{\times}$ (respectively $Z(A)^{\times}$)
for the group of  invertible elements of~$A$ (respectively, of~$Z(A)$). $T^{\rm sym}$
(respectively $T^{\rm asym})$ stands for the set of symmetric (respectively
antisymmetric) elements of the subset~$T$ of~$A$, i.e., the set of elements~$a$ in $T$ such that $a^{\ast}=a$ (respectively $a^{\ast}=-a$). The involution~$\ast$ induces an involution on the ring of $2\times2$ matrices with
coef\/f\/icients in~$A$, by $(g^{\ast})_{ij}=g_{ji}^{\ast}$ ($g\in \M(2,A)$), which
we denote also with the symbol $\ast$.

\subsection[The groups $\SL_{*}^{1}(2,A)$]{The groups $\boldsymbol{{\rm SL}_{*}^{1}(2,A)}$}

We give a brief description of the groups $\SL_{\ast}^{1}(2,A)$. For more details see~\cite{P-SA}.

If $A$ is a unitary ring  with an involution $\ast,$ and $J=\left(
\begin{matrix}%
0 & 1\\
1 & 0
\end{matrix}
\right)$ in $\M(2,A)$, let
$
\GL_{\ast}^{1}(2,A):=\left\{
g=\begin{pmatrix}
a & b\\
c & d
\end{pmatrix}\in \M(2,A)\colon g^{\ast}Jg=\lambda(g)J,\;\lambda(g)\in (Z(A)^{\rm sym})^{\times}\right \}$.
Then $\GL_*^1(2,A)$ is a group.  Mo\-re\-over, the map \begin{gather*}
\deter\colon \ \GL_{\ast}^{1}(2,A)\to (Z(A)^{\rm sym})^{\times}
\end{gather*}
given by $\deter(g)=\lambda(g)=ad^{\ast}+bc^{\ast
}=a^{\ast}d+c^{\ast}b$ is an homomorphism.

We def\/ine $\SL_{\ast}^{1}(2,A)$ as the kernel of $\deter$.
One can observe that the entries of  $g=\begin{pmatrix}
a & b\\
c & d
\end{pmatrix} \in \SL_{\ast}^{1}(2,A)$ satisfy:
$a^{\ast }c=-c^{\ast}a$, $ab^{\ast}=-ba^{\ast}$, $b^{\ast
}d=-d^{\ast}b$, $cd^{\ast}=-dc^{\ast}$.

\subsection{The involutive ring of truncated polynomials}

Let $k$ be a f\/inite f\/ield of $q$ elements, where $q$ is a power of an odd prime~$p$. We consider~$K$ the unique quadratic extension of~$k$ and we take $\Delta\in K$ such that $K=k(\Delta)$ and $\Delta^2\in k$. We write $\overline{a+b\Delta}$ to denote the image $a-b\Delta$  of  the element $a+b\Delta$ under the nontrivial element of the Galois group of the extension~$K/k$.  Let
\begin{gather*}
A_n=K[x]/\langle x^n\rangle, \qquad n\in\mathbb{N},
\end{gather*}
which will be considered as  polynomials (with coef\/f\/icients in $K$), truncated at $n$ (i.e., such that $x^m=0$ for $m\geqq n$).

We def\/ine an involution $*$ in the $k$-algebra $A_n$ by
\begin{align*}
a+b\Delta&\mapsto  \overline{a+b\Delta},\\
x&\mapsto  -x.
\end{align*}

 We f\/irst present  some results concerning cardinalities about sets that we will use later on.
 \begin{proposition}\label{card}
If $ \mid S \mid$ denotes the cardinality of $S$, we have
\begin{enumerate}\itemsep=0pt
\item[$1.$]
The order of the group of invertible elements of $A_n$ is $\vert A_n^{\times}\vert=(q^2-1)q^{2(n-1)}$.
\item[$2.$]
 \begin{gather*}
 A_n^{\rm sym}=\left\{a=\sum_{i=0}^{n-1} a_ix^i\colon  a_{2i}\in k \text{ and } a_{2i+1}\in \Delta
 k\right\}
\end{gather*}
 has cardinality   $\vert A_n^{\rm sym}\vert =q^n$.
\item[$3.$]
\begin{gather*}
A_n^{\rm asym}=\left\{a=\sum_{i=0}^{n-1} a_ix^i\colon  a_{2i}\in \Delta k \text{ and } a_{2i+1}\in  k\right\}
\end{gather*}
has cardinality  $\vert A_n^{\rm asym}\vert =q^n$.
\end{enumerate}
\end{proposition}

\begin{proof}
1.~The result is clear observing that  an invertible element of the ring must have nonzero constant term.

2.~Let $a=\sum\limits_{i=0}^{n-1}a_ix^i$ be an arbitrary element in $A_n$, then
\begin{gather*}
a^*=\sum_{i=0}^{n-1} (-1)^i \overline{a_i}x^i.
\end{gather*}
So $a^*=a$ if and only if  $a_i=(-1)^i\overline{a_i}$, for each $i$. Thus
\begin{gather*}
 A_n^{\rm sym}=\left\{a=\sum_{i=0}^{n-1} a_ix^i\colon  a_{i}\in  k \text{ for } i \text{ even and } a_{i}\in \Delta k \text{ for }i \text{ odd}\right\},
  \end{gather*}
  and the result follows.

3.~Similar to 2. This completes the proof.
\end{proof}

\begin{proposition}\label{orderSL}
The order of $\SL_{\ast}^{1}(2,A_n)$ is $(q^2-1)q^{4n-3}(q+1)$.
\end{proposition}

\begin{proof}
The group $\SL_*(2,A_n)$ acts on $\M_{2\times 1}(A_n)$ by left multiplication. Set
\begin{gather*}
		O_1   =   \left\{\begin{pmatrix}a\\ua\end{pmatrix} \in \M_{2\times 1}(A_n)\colon a \in A_n^{\times}, u\in A_n^{\rm asym} \right\}, \\
	O_2   =   \left\{\begin{pmatrix}uc\\c\end{pmatrix} \in \M_{2\times 1}(A_n)\colon c \in A_n^{\times}, u\in A_n^{\rm asym} \backslash A_n^{\times} \right\}.
\end{gather*}
We claim that the orbit  ${\rm Orb}_{\SL_*^1(2,A_n)}\begin{pmatrix}1\\0\end{pmatrix}$ is the union of $O_1$ and $O_2$. In fact,  we note that the f\/irst column  $\begin{pmatrix} a\\ c\end{pmatrix}$ of a matrix in $\SL_*^1(2,A_n)$ satisf\/ies   $a^*c=-c^*a$. Since $a$ or $c$ must be invertible, we get  $ca^{-1}$ (or  $ac^{-1}$) is anti-symmetric, then  $c=ua$ (or $a=uc$), for some anti-symmetric element $u$. From here  ${\rm Orb}_{\SL_*^1(2,A_n)}\begin{pmatrix} 1\\0\end{pmatrix}$ is contained in the union $O_1	\cup O_2$. On the other hand,  if $\begin{pmatrix}a\\ua\end{pmatrix}\in O_1$, then
\begin{gather*}
\begin{pmatrix}
a & 0\\
ua & a^{*^{-1}}
\end{pmatrix} \in \SL_*^1(2,A_n),\qquad
\begin{pmatrix} a\\ua\end{pmatrix} = \begin{pmatrix}
a & 0\\
ua & a^{*^{-1}}
\end{pmatrix}\begin{pmatrix} 1\\0\end{pmatrix}.
\end{gather*}
Thus,  $O_1$ is contained in the orbit ${\rm Orb}_{\SL_*^1(2,A_n)}\begin{pmatrix} 1\\0\end{pmatrix}$. Similarly, ${\rm Orb}_{\SL_*^1(2,A_n)}\begin{pmatrix} 1\\0\end{pmatrix}$ contains $O_2$. This proves the claim.

Now given that the sets are disjoint we verify
\begin{gather*}
\left\vert {\rm Orb}_{\SL_*^1(2,A_n)}\begin{pmatrix} 1\\ 0\end{pmatrix}\right\vert=\big(q^2-1\big)q^{2(n-1)}q^n+\big(q^2-1\big)q^{2(n-1)}q^{n-1}=\big(q^2-1\big)q^{3n-3}(q+1).
\end{gather*}
 Finally,  the isotropy group ${\rm Stab}\left(\begin{pmatrix} 1\\0\end{pmatrix}\right)$ is the group of upper unipotent matrices in $\SL_*^1(2,A_n)$ which  has cardinality  $q^n$, and therefore  our proposition follows.
\end{proof}

\section[Bruhat presentation for ${\rm SL}_{*}^{1}(2,A_n)$]{Bruhat presentation for $\boldsymbol{{\rm SL}_{*}^{1}(2,A_n)}$}\label{section3}

 In this section, we prove that the group $G=\SL_{\ast}^{1}(2,A_n)$  has a Bruhat-like presentation, which will be used in the construction of a Weil representation of~$G$.

 To this end, we set
\begin{gather*}
h_{t}=h(t)=\left(
\begin{matrix}
t & 0\\
0 & t^{\ast-1}
\end{matrix}
\right), \quad t\in A^{\times}, \qquad
 w=\left(
\begin{matrix}
0 & 1\\
1 & 0
\end{matrix}
\right),
\end{gather*}
  and
  \begin{gather*}
   u_{s}=u(s)=\left(
\begin{matrix}
1 & s\\
0 & 1
\end{matrix}
\right), \quad s\in A^{\rm asym},
\end{gather*}
noticing  that the matrices $h_{t}, \omega, u_{s}\in \SL_{\ast}^{1} (  2,A )$.

\begin{lemma}\label{inv1}
Let $a$ and $c$ be two elements in $A_n$ such that $Aa+Ac=A$
 and $a^*c=c^*a$. Then, there exits an element $s\in A_n^{\rm asym}$ such that $a+sc$ is invertible.
\end{lemma}

\begin{proof}
We  f\/irst observe that $a$ or $c$ has to be invertible. Suppose that $a$ is invertible, then considering $s=0$ we have the result. On the other hand, if~$a$ is non-invertible, then~$c$ is invertible and so any nonzero element~$s$ in~$k\Delta$ proves the lemma.
\end{proof}

\begin{proposition}\label{gen}
The matrices   $h_t$, $u_s$ and $w$, with $t\in A_n^{\times}$, $s\in A_n^{\rm asym}$,   generate $\SL_*(2,A_n)$.
\end{proposition}

\begin{proof}
Let $g=\begin{pmatrix} a&b\\c&d\end{pmatrix}\in \SL_*^1(2,A_n)$.
If $c=0$, then   $g=h_au_{a^{-1}b}$.

Suppose now that $c$ is invertible. Since $\deter(g)=1$, i.e., $ad^*+bc^*=1$ one gets $b=c^{*^{-1}}+ac^{-1}d$. One checks  $g=h(c^{*^{-1}})u(c^*a)wu(c^{-1}d)$.

Now if $c\neq 0$ and non-invertible, we take an antisymmetric element $s$ satisfying the above lemma. Now
\begin{gather*}
wu_sg=
\begin{pmatrix} 0&1\\1&0\end{pmatrix}\begin{pmatrix} 1&s\\0&1\end{pmatrix}\begin{pmatrix} a&b\\c&d\end{pmatrix}=\begin{pmatrix} c&d\\a+sc&b+sd\end{pmatrix}.
\end{gather*}
Thus the matrix $wu_sg$ has entry $(2,1)$ invertible and  we can write
\begin{gather*}
wu(s)g=h((a+sc)^{*^{-1}})u((a+sc)^*c)wu((a+sc)^{-1}(b+sd)).
\end{gather*}
From here,  $\SL_*^1(2,A_n)$ is generated by the matrices  $h_t$, $u_s$ and $w$ with $t\in A_n^{\times}$, $s\in A_n^{\rm asym}$.
\end{proof}

\begin{lemma}\label{inv2}
Let $a$ and $b$ be two non-invertible antisymmetric  elements  in $A_n$. Then,  we can find an antisymmetric invertible element~$v\in A_n$ such that $a-v^{-1}$ and $b+v$  are antisymmetric invertible elements in~$A_n$.
\end{lemma}

\begin{proof}
Since $a$, $b$ are non-invertible elements, they are in the ideal generated by $x$. Then taking  any nonzero element $v\in \Delta k$, we check that  $a-v^{-1}$ and $b+v$  are antisymmetric invertible elements in~$A_n$.
\end{proof}

Lemmas \ref{inv1}, \ref{inv2} and Proposition \ref{gen} prove, using the same argument as in Theorem 15 of \cite{P}, our next result:

\begin{theorem}\label{presentacion}
The matrices $h_t$, $u_s$ and $w$  $(t\in A_n^{\times}$, $s\in A_n^{\rm asym})$, with the commutating relations:
\begin{gather*}
h_{t_1}h_{t_2}=h_{t_1t_2},\\
u_{s_1}u_{s_2}=u_{s_1+s_2},\\
h_tu_s=u_{tst^*}h_t,\\
 w^2=1,\\
h_tw=wh_{t^{*^{-1}}},\\
wu_{t^{-1}} wu_{- t}wu_{t^{-1}}=h_{-t}
\end{gather*}
give a presentation of $\SL_*^1(2,A_n)$.
\end{theorem}

\section[A generalized Weil representation of ${\rm SL}_*^1(2,A_n)$]{A generalized Weil representation of $\boldsymbol{{\rm SL}_*^1(2,A_n)}$}\label{section4}

In \cite[Theorem~4.4]{G-P-SA}, a Weil representation is constructed for groups that af\/ford a Bruhat presentation as the one obtained in Theorem~\ref{presentacion}.  Specif\/ic data  are required.

First, we must have a f\/inite right $A$-module $M$ and pair of functions $\chi\colon M\times M\rightarrow\mathbb{C}^{\times}$ and $\gamma\colon A^{\rm asym}\times M\rightarrow\mathbb{C}^{\times}$, a~nonzero complex number $c$, and a character $\alpha\in\widehat{A^{\times}}$, satisfying the following properties and relations among them:
\begin{enumerate}\itemsep=0pt
\item [a)] $\chi$ is bi-additive;
\item [b)] $ \chi(mt, v) = \alpha(tt^{\ast})\chi(m, vt^\ast) $ for $m,v\in M$ and $t\in
A^{\times}$;
\item[c)] $\chi(v,m)=[\chi(m,v)]^{-1}$;
\item [d)] $\chi(v,m)=1$ for all $m\in M$, implies $v=0$;
\item [e)]  $\gamma(b,mt)=\gamma(tbt^{\ast},m)$;
\item [f)] $\gamma(t,m+z)=\gamma(t,m)\gamma(t,z)\chi(m,zt)$ for all   $m,z \in M$, $t\in A^{\rm asym}$,
 where $t$ is anti-symmetric invertible in $A$ and  $c\in\mathbb{C}^{\times}$ satisf\/ies\\ $c^{2}\left\vert M\right\vert =1$;
\item [g)] $\gamma(b+b^{\prime},m)=\gamma(b,m)\gamma(b^{\prime},m)$, for all $b, b'\in A^{\rm asym}$ and $m\in M$, and
\item[ h)]  $\underset{m\in M}{\sum}\gamma(t,m)=\frac{\alpha(t)}{c}$.
\end{enumerate}

Then, with the above data we have (see \cite{G-P-SA}):

\begin{theorem}\label{theorem}
 If $\SL_*^{1}(2,A)$ has a Bruhat presentation, the data    $(M,\chi,\gamma,c)$  defines   a $($linear$)$ representation $(\mathbb{C}^{M},\rho)$ of $\SL_*^{1}(2,A)$,  which we call Weil representation,  by
 \begin{enumerate}\itemsep=0pt
\item[$1)$] $\rho_{u_b}(e_a)=\gamma(b,a)e_a$,
\item[$2)$] $\rho_{h_t}(e_a)=\alpha(t)e_{at^{-1}}$,
\item[$3)$] $\rho_w(e_a)=c\sum\limits_{b\in M}\chi(-a,b)e_b$,
 \end{enumerate}
for $a\in M$, $b\in A^{\rm asym}$, $t\in A^{\times}$ and $e_a$ the Dirac function at~$a$, defined by $e_a(u)=1$ if $u=a$ and $e_a(u)=0$ otherwise.
\end{theorem}

\subsection[Construction of  data for ${\rm SL}_*^1(2,A_n)$]{Construction of  data for $\boldsymbol{{\rm SL}_*^1(2,A_n)}$}\label{data}

Since we already know that  $\SL_*^1(2,A_n)$ has a Bruhat presentation, we construct now  a data for this group, in order to apply Theorem~\ref{theorem}.

Let $\psi_{0}$ be a nontrivial  additive  character of $K$  such that~$\psi_0$ is nontrivial in~$k$ and in~$\Delta k$.
We consider the biadittive function $\chi$ from $A_n\times A_n$ to~$\mathbb{C}^{\times}$ given by
\begin{gather*}
(a,b)\mapsto \psi(a^*b),
\end{gather*}
 where $\psi$ is the nontrivial character of~$A_n$ def\/ined as
\begin{gather*}
\psi\big(a_{o}+a_{1}x+\dots+a_{n-1}x^{n-1}\big)=\psi_{0}(a_{n-1}).
\end{gather*}
 We take $M=A_n$ and we assume $\alpha={\bf 1}$.
It is clear that a), b) and c) above are fulf\/illed. We prove now~d).

\begin{proposition}
The biadditive map $\chi$ is non-degenerate.
\end{proposition}
\begin{proof}
 Let $a$ be a nonzero element of $A_n$. We need to prove that there is an element~$b$ in~$A_n$ such that~$\chi(a,b)\neq 1$. Let us write $a=a_{o}+a_{1}x+\dots+a_{n-1}x^{n-1}$. If $a_i$ is the f\/irst nonzero coef\/f\/icient  of~$a$, set $b=tx^{n-i-1}$. Then
$\chi(a,b)=\psi_0((-1)^{i} \bar a_{i}t)$. If $t$ runs over $K$, then so  does $(-1)^{i}\bar a_{i}t$. Since the character $\psi_0$ is nontrivial, the result follows.
\end{proof}

We now def\/ine the function $\gamma$ by
$\gamma(t,m)=\chi(-2^{-1}tm,m)$, for $t\in A_n^{\rm asym}$ and $m$ in $A_n$.
We claim  that $\gamma$ satisf\/ies  conditions  e), f), g), and  h) above.

e)  Let $b\in A_n^{\rm asym}$, $m\in A_n$ and $t\in A_n^{\times}$. Then
\begin{gather*}
\gamma\big(tbt^*,mt^{-1}\big)= \chi\big({-}2^{-1}tbt^*mt^{-1},mt^{-1}\big)=\chi\big({-}2^{-1}bt^*m,mt^{-1}\big).
 \end{gather*}
 Now, using that $\chi$ is balanced, we get   $\gamma(tbt^*,mt^{-1})=\chi(-2^{-1}bm,m)=\gamma(b,m)$.

f) Let $t\in A_n^{\rm asym}$ and $m,v\in A_n$.  Then
\begin{gather*}
\gamma(t,m+v)=\chi\big({-}2^{-1}t(m+v),m+v\big)\\
\hphantom{\gamma(t,m+v)}{} =
 \chi\big({-}2^{-1}tm,m\big)\chi\big({-}2^{-1}tv,v\big)\chi\big({-}2^{-1}tm,v\big)\chi\big({-}2^{-1}tv,m\big).
\end{gather*}
Now, since $\chi$ is balanced, we have $\chi(-2^{-1}tm,v)=\chi(m,2^{-1}tv)$. On the other hand,   we have
$\chi(-2^{-1}tv,m)=\chi(m,2^{-1}tv)$.  Then $\chi(-2^{-1}tm,v)\chi(-2^{-1}tv,m)=\chi(m,vt)$, and
\begin{gather*}
\gamma(t,m+v)=\chi\big({-}2^{-1}tm,m\big)\chi\big({-}2^{-1}tv,v\big)\chi(m,vt).
\end{gather*}

 g)   $\gamma(b_1+b_2, m)=\gamma(b_1, m)\gamma(b_2, m)$ follows from the bi-additive property of~$\chi$.

 h)  Set $c=(-1)^n\frac{1}{q^n}$. Since $\vert A_n\vert=q^{2n}$, we have $c^2\vert A_n\vert=1$. The result will follow from Proposition~\ref{gauss} below, which needs Lemma~\ref{norm}:

\begin{lemma}\label{norm}
Let $\psi$ be the additive  character of $K$ defined above. Then
\begin{gather*}
\sum_{z\in K}\psi(\n_{K/k}(\lambda z))=\begin{cases} q^2 & \text{if} \ \ \lambda=0,  \\   -q  & \text{if} \ \ \lambda\neq 0.
\end{cases}
\end{gather*}
\end{lemma}

\begin{proof}
The case $\lambda=0$ is clear. Suppose then $\lambda \neq 0$. Since $z $ runs over $K$ it  suf\/f\/ices to compute the sum $\sum\limits_{z\in K}\psi(\n_{K/k}( z))$. We observe
\begin{gather*}
K=\bigcup_{t\in k}\n_{K/k}^{-1}(t).
\end{gather*}
Since the character $\psi$ is nontrivial on~$k$ and the cardinality of $\n_{K/k}^{-1}(t)$ is $q+1$ for any $t\in k^{\times}$, we have
\begin{gather*}
\sum_{z\in K}\psi(\n_{K/k}(\lambda z))=\psi(0) + \sum_{t\in k^{\times}}(q+1)\psi(t)=1+(q+1)(-1)=-q.\tag*{\qed}
\end{gather*}\renewcommand{\qed}{}
\end{proof}

Now,
 \begin{proposition}\label{gauss}  We have
\begin{gather*}
\sum_{y\in A_n}\gamma(t,y)=(-1)^n q^n.
\end{gather*}
\end{proposition}

\begin{proof}
By def\/intion $\gamma (t,y)=\chi (-{\frac{1}{2}} ty,y)$.
We set $y=\sum\limits_{i=0}^{n-1}\alpha_{i}x^i$ and $-\frac {1}{2}t=\sum\limits_{i=0}^{n-1}\lambda_{i}x^i$ for $y\in{A_n}$ and $t\in A^{{\rm asym}^{\times}}$.  We will write sometimes $(y)_i$ for the coef\/f\/icient of~$x^i$ in~$y$.
 Then
\begin{gather*}
\sum _{y\in{A_n}}\gamma(t,y)=\sum_{y\in A_n }\chi \left(-{\frac{1}{2}} ty,y\right)=
\sum_{y\in A_n}\psi\left(-\frac{1}{2}ty^{\ast}y\right) =
\sum_{i=1}^{n}\psi_{0}\left(\left(-\frac{1}{2}ty^{\ast}y\right)_{n-1}\right).
\end{gather*}
Since $t^{\ast}=-t$ we have $\lambda_i=d_i\in k$ for $i$ odd, and
$\lambda_i=d_i\Delta \in k{\Delta }$ for  $i $ even.

We will split the proof into two cases, according to $n$ is even or odd.
We f\/irst assume   that $n$ is even.
Using  the above, and after a computation, we have
\begin{gather}\label{eq1}
 \sum_{y\in{A_n}}\gamma(t,y) = \sum_{\alpha_{i},i=0}^{n-1}
\psi_{0}(  d_{n-1}\beta_{0}+\Delta d_{n-2}\beta_{1}+\dots+ d_{1}\beta_{n-2}+ \Delta d_{ 0}\beta_{n-1}  ) ,
\end{gather}
where $\beta_i= \sum\limits_{j=0}^{i} (-1)^{j}\bar {\alpha}_j\alpha_{i-j}$.

We f\/irst observe  that $\alpha_{n-1}$ appears only  in $\beta_{n-1}$. Summing f\/irst over $\alpha_{n-1}$, we write
\begin{gather*}
\sum \cdots \sum\limits_{\alpha_0,\alpha_{n-1}}\cdots \psi_0(\Delta d_0(\bar \alpha_0\alpha_{n-1}-\bar \alpha_{n-1}\alpha_0)
\\
\qquad {}=
 \sum\cdots  \sum_{\alpha_0,\alpha_{n-1}}\cdots
 \psi_0(\Delta d_0(\bar \alpha_0\alpha_{n-1}-
\overline{\overline  {\alpha_{0}}{\alpha_{n-1}}})).
\end{gather*}

Given the assumption on $\psi_0,$ the character that sends
$\alpha_{n-1}$ into $\psi_0( \Delta d_{0}(\bar \alpha_0\alpha_{n-1}-\overline{\bar \alpha_0\alpha_{n-1}}))$
 is  an additive,  nontrivial character of~$K$. This sum is zero unless  $\alpha_0=0$. Then the sum over $\alpha_0$ and $\alpha_{n-1}$ contributes with~$q^2$.
In this way, the above sum becomes
\begin{gather*}
\sum\cdots \sum\limits_{\alpha_1,\alpha_{n-2}}\cdots q^2\psi_0(d_{n-3}\gamma_1+\Delta d_2\gamma_2+\cdots +d_1\gamma_{n-3}+\Delta d_0\gamma_{n-2}),
\end{gather*}
 where $\gamma_i=\sum\limits_{j=1}^{i}(-1)^j\bar \alpha _j\alpha_{i-j+1}$.
We note again that, in this case, $\alpha_{n-2}$ appears only in $\gamma_{n-2},$ and we can write the sum as
\begin{gather*}
\sum\cdots \sum\limits_{\alpha_1,\alpha_{n-2}}\cdots q^2\psi_0(\Delta d_0(-\bar \alpha_1\alpha_{n-2}+\bar \alpha_{n-2}\alpha_1)).
\end{gather*}
We sum next over $\alpha_1$ and $\alpha_{n-2},$ and we continue with this process
  to obtain at the end that~\eqref{eq1} is $(q^2)^{ \frac{n}{2}}=q^{n}$.

We assume this time $n$ is odd. In this case, we have
\begin{gather}\label{eq2}
 \sum_{y\in{A_n}}\gamma(t,y)=\sum_{\substack{\alpha_{i}\\i=0,\dots,n-1}}
\psi_{0}( \Delta d_{n-1}\beta_{0}+ d_{n-2}\beta_{1}+\dots+\Delta d_{1}\beta_{n-2}+  d_{ 0}\beta_{n-1}  ).
\end{gather}

As before, $\alpha_{n-1}$ appears only  in  $\beta_{n-1}$. Summing  f\/irst over $\alpha_0$ and $\alpha_{n-1},$ to get
\begin{gather*}
\sum\cdots \sum\limits_{\alpha_0,\alpha_{n-1}}\cdots \psi_0(\Delta d_0(\bar \alpha_0\alpha_{n-1}+\bar \alpha_{n-1}\alpha_0)\\
\qquad {}=\sum\cdots \sum_{\alpha_0,\alpha_{n-1}}\cdots \psi_0(\Delta d_0(\bar \alpha_0\alpha_{n-1}+\overline{\overline  {\alpha_{0}}{\alpha_{n-1}}}).
\end{gather*}
We have $\psi_0(\Delta d_0(\bar \alpha_0\alpha_{n-1}+\overline{\overline  {\alpha_{0}}{\alpha_{n-1}}})
=\psi_0(\Delta d_0 \operatorname{Tr}((\bar \alpha_0\alpha_{n-1}))$, where ${\rm Tr}$ stands for the f\/ield trace of the extension $E \supset F$.

Arguing as before,  at the last step, i.e., after
$\frac{n-1}{2}$ steps
we get
\begin{gather*}
\big(q^2\big)^ { \frac{n-1}{2}   } \sum\limits_{\alpha_{\frac{n-1}{2}}}\psi_0\big(\Delta \overline{\alpha_{\frac{n-1}{2}}}{\alpha_{\frac{n-1}{2}}}\big).
\end{gather*}
Using Lemma~\ref{norm}, we obtain \eqref{eq2} is equal to~$-q^n$.
\end{proof}

 \section[A first decomposition]{A f\/irst decomposition}\label{section5}

We will get a f\/irst decomposition of~$\rho$, constructed in Theorem~\ref{theorem}. To this end, we lean on Theorem~7.6 in~\cite{G-P-SA}. The unitary group  $U=U(\chi,\gamma,c)$ consisting of all $A_n$-linear automor\-phism~$\varphi$ of~$M$ such that $\gamma(b,\varphi(x))=\gamma(b,x)$, for any $b\in A_n$ and $x\in M $, allows us to obtain a decomposition of the Weil representation. In fact, the characters of~$U$ def\/ine the invariant subspaces of a~decomposition of~$\rho$. So, we will devote ourselves to obtain the structure of this group. Observing f\/irst that $U$ is abelian, Theorem~7.6 in~\cite{G-P-SA} reads as:

If $\Lambda \in \widehat{U}$, let $W_{\Lambda}$ be the vector subspace of $W$ of the $\Lambda$-homogeneous functions, that is,  the vector subspace of the functions $f\in W$ such that $f(ua)=\Lambda(u)f(a)$, for $a\in A_n$ and $u\in U$, then

\begin{theorem}
The Weil representation $(W,\rho)$ is the direct sum of all  $W_{\Lambda}$, where $\Lambda$ runs  over all linear characters of~$U$.
\end{theorem}

We f\/irst prove:

\begin{proposition}  \label{action}
We have
\begin{enumerate}\itemsep=0pt
\item[$1.$] The group $A_n^{\times}$ acts transitively on $ A_n^{\rm sym}\cap A_n^{\times}$  by $a\cdot t=ata^*$.
\item[$2.$] The group of units $A_n^{\times}$   acts transitively on $ A_n^{\rm asym}\cap A_n^{\times}$ under the same action.\end{enumerate}
\end{proposition}

\begin{proof}
1.~We  consider the case when $n$ is even.  Since the ring $A_n$  is commutative,  we see that if $t_1$, $t_2$ are in the orbit ${\rm Orb}(1)$, then $t_1t_2\in {\rm Orb}(1)$.  We will f\/irst prove that every element of the form $t=1+a_1\Delta x+a_2x^2+\cdots +a_{n-1}\Delta x^{n-1}$  ($a_i\in k$, for all~$i$) belongs to  ${\rm Orb}(1)$.

In fact, one can check that  there is an element $b=1+b_1\Delta x+b_2x^2+\cdots +b_{n-1}\Delta x^{n-1}\in A_n$ $(b_i \in k),$ such that $t=bb^*$,  given that the system
\begin{gather*}
1=1,\\
2b_1= a_1\\
2b_2+b_1^2\Delta^2=a_2,\\
2b_3+2b_1b_2=a_3,\\
2b_4+2b_1b_3\Delta^2+b_2^2=a_4,\\
\cdots\cdots\cdots\cdots\cdots\cdots\cdots\cdots\\
2b_{n-1}+2b_1b_{n-2}+2b_2b_{n-3}+\cdots+2b_{n-2}b_2=a_{n-1}
\end{gather*}
has always a solution.

Now, for an arbitrary invertible symmetric element
$t=a_0+a_1\Delta x+a_2x^2+\cdots +a_{n-1}\Delta x^{n-1}$, given that any nonzero element of $k$ is in ${\rm Orb}(1)$,     we have that
\begin{gather*}
t=a_0\big(1+a_0^{-1}a_1\Delta x+a_0^{-1}a_2x^2+\cdots +a_0^{-1}a_{n-1}\Delta x^{n-1}\big)
\end{gather*}
 belongs to ${\rm Orb}(1)$ and therefore ${\rm Orb}(1)=A_n^{\rm sym}\cap A_n^{\times}$.

The case when $n$ is odd is handled in a similar way.

2.~Notice that $\Delta a\in A_n^{\rm asym}\cap A_n^\times$, for any symmetric invertible element $a\in A_n$. Then, it follows from part~$1$ that  $\Delta (A_n^{\rm sym}\cap A_n^{\times} )=\Delta {\rm Orb}(1)$ is a subset of   $A_n^{\rm asym}\cap A_n^{\times}$ that has the same cardinality than $A_n^{\rm asym}\cap A_n^{\times}$. We have that the orbit of~$\Delta$ is the unique orbit for the action.

This proves the proposition.
\end{proof}

 Next we start to prove that the group $U=U(\chi,\gamma,c)$ is isomorphic to the group of all $a\in  A_n^{\times}$ such that   $aa^*=1$. We f\/irst prove
\begin{lemma}\label{order}
The subgroup $ \{a \in A_n\colon aa^*=1\} $ has $(q+1)q^{n-1}$ elements.
\end{lemma}

\begin{proof}
By proof of Proposition \ref{action}, the cardinality of  orbit ${\rm Orb}(1)$ under the action of $A_n^{\times}$  on
$A_n^{\rm sym}\cap A_n^{\times}$ is $(q-1)q^{n-1}$. The isotropy group ${\rm Stab}(1)$ is~$U$, hence
\begin{gather*}
 \vert \{a \in A_n \colon aa^*=1\}\vert =\dfrac{\vert A_n^{\times}\vert}{\vert {\rm Orb}(1)\vert}=\dfrac{(q^2-1)q^{2(n-1)}}{(q-1)q^{n-1}}=(q+1)q^{n-1}.
 \tag*{\qed}
 \end{gather*}
\renewcommand{\qed}{}
\end{proof}

\begin{proposition}
$U(\chi,\gamma,c)\cong \{a \in A_n \colon aa^*=1\}$.
\end{proposition}
\begin{proof}

We have  $U(\chi,\gamma,c)$ consists of all  $A_n$-automorphisms $\varphi \colon A_n \rightarrow A_n$  such that  $\gamma(b,\varphi(y))=\gamma(b,y)$ for all $b\in {A_n^{\rm asym}}^\times$ and  $y\in A_n$. Hence  $\varphi(1)$ determines completely $\varphi$.

Now, the def\/inition of $\gamma$ implies
$\psi(-\frac{1}{2}b\varphi(y)\varphi(y)^*)=\psi(-\frac{1}{2}byy^*)$ and so
$\psi(-\frac{1}{2}byy^*\varphi(1)\varphi(1)^*)$ $=\psi(-\frac{1}{2}byy^*)$.
Setting $t=-\frac{1}{2}b $ and $d=\varphi(1)\varphi(1)^*$ we have in the notations of Section~\ref{data}
\begin{gather}
 \psi(dtyy^*)=\psi(tyy^*) \label{eq1+}
\end{gather}
for all $t\in {A_n^{\rm asym}}^\times$ and  $y\in A_n$,
from which
$\psi_0((dtyy*)_{n-1})=\psi_0((tyy^*)_{n-1})$.
We will prove that $d=1$. To do that set $d=d_0+d_1x+\dots+d_{n-1}x^{n-1}$.
Given that $d$ is symmetric, $d_i \in k$ for $i$  even, and $d_i \in k\Delta$ for $i$ odd. Now, since   $t$ is antisymmetric, we see
$t=\lambda_0\Delta+\lambda_1x+\lambda_2\Delta x^2+\lambda_3x^3+\lambda_4\Delta x^4+\cdots)$ with $\lambda_i \in k$ and $\lambda_0\neq 0$.
We write $y=y_0+y_1x+\cdots+y_{n-1}x^{n-1}$. We want to prove then  $d_0=1$ and $d_i=0$ for $i>0$.
Now, the equation~\eqref{eq1+} implies
\begin{gather*}
\psi_0(d_0\lambda_0\Delta(\bar y_0y_{n-1}-\bar y_1y_{n-2}+\cdots +(-1)^{n-1}\bar y_{n-1}y_0)+ (d_0\lambda_1(\bar y_0y_{n-2}-\bar y_1y_{n-3}+ \cdots\\
\qquad\quad{} +(-1)^{n-2}\bar y_{n-2}y_0)+\cdots + d_{n-2}\lambda_1\bar y_0y_0+d_{n-1}\lambda_0\Delta \bar y_0y_0) \\
\qquad{} =\psi_0(\lambda_0\Delta(\bar y_0y_{n-1}-\bar y_1y_{n-2}+\cdots +(-1)^{n-1}\bar y_{n-1}y_0)+\cdots\\
\qquad\quad{} +\lambda_{n-2}\Delta^{\delta_1}(\bar y_0y_1-\bar y_1y_0)+  \lambda_{n-1}\Delta^{\delta_2})
\end{gather*}
 ($\delta_1=1$ if $n$ is even, and $0$ otherwise; $\delta_2=0$ if $n$ is even, and $1$ otherwise).

If we take $\lambda_0= 1$, $\lambda_i=0$ if $i>0$, and taking $y_0$ arbitrary, and $y_i=0$ for $i>0$ we end up with
  $\psi_0(d_{n-1}\Delta \bar y_0y_0)=1$.  By the choice of~$\psi_0$, it follows that $d_{n-1}= 0$.
 For the  next step (to prove now that $d_{n-2}=0$)  we take $\lambda_0=\lambda_1=1$ and $\lambda _i=0$  for $i>1$.
  As before, we set $y_i=0$ for $i>0$ and $y_0$ arbitrary, getting this  time
 $\psi_0(d_{n-2}\Delta \bar y_0y_0)=1$. So, $d_{n-2}=0$.

 Continuing with this process, we obtain $d_i=0$ for all $i>0$.
 Finally, considering $\lambda_0=1$ and $\lambda_i=0$ for $i>0$, we are left with
\begin{gather*}
\psi_0\big(d_0\Delta\big(\bar y_0y_{n-1}-\bar y_1y_{n-2}+\cdots +(-1)^{n-1}\bar y_{n-1}y_0\big)\big)\\
\qquad{}=
 \psi_0\big(\Delta\big(\bar y_0y_{n-1}- \bar y_1y_{n-2}+
\cdots +(-1)^{n-1}\bar y_{n-1}y_0\big)\big),
\end{gather*}
from which it follows $d_0=1$.
\end{proof}

\begin{proposition}  \label{uo}
We have
\begin{enumerate}\itemsep=0pt
\item[$1.$] $U\simeq \n_{K/k}^{-1}(1)\times U_0$ where $U_0=\{b\colon b\in U,\, (b)_0=1\}$,  $(b)_0$ as defined in Proposition~{\rm \ref{gauss}}.

\item[$2.$]
\begin{itemize}\itemsep=0pt
\item  If $n$ is even, then the group $U_0$  consists of all elements $z$ of the form:
\begin{gather*}
z=1 +\lambda_1x+(f_1(\lambda_1)+\lambda_2\Delta)x^2+(\lambda_3+f_2(\lambda_1,\lambda_2)\Delta)x^3+\cdots\\
 \hphantom{z=}{} +(\lambda_{n-1}+f_{n-2}(\lambda_1,\dots,\lambda_{n-2})\Delta)x^{n-1},
\end{gather*}
 with $\lambda_i\in k$.
\item  If $n$ is odd,   then the group $U_0$  consists of all elements $z$ of the form:
\begin{gather*}
z=1  +\lambda_1x+(f_1(\lambda_1)+\lambda_2\Delta)x^2+(\lambda_3+f_2(\lambda_1,\lambda_2)\Delta)x^3+\cdots\\
\hphantom{z=}{} +(f_{n-2}(\lambda_1,\dots,\lambda_{n-2})+\lambda_{n-1}\Delta)x^{n-1},
\end{gather*}
where $\lambda_i\in k$.
\end{itemize}

\end{enumerate}
\end{proposition}
\begin{proof} Part 1 follows directly from the def\/initions.

 We prove  part~2.
If $b=r+c\Delta$, we will write $r=\operatorname{Real}(b)\in k$ and $c=\operatorname{Im}(b)\in k$.
  Expanding
\begin{gather*}
\big(1+b_1x+b_2x^2+\cdots+b_{n-1}x^{n-1}\big)\big(1-\bar{b}_1x+\bar{b}_2x^2-\cdots\pm \bar{b}_{n-1}x^{n-1}\big)=1,
\end{gather*}
    we see  that $b_1\in k$. We set $\lambda_1=b_1$.

   Similarly, we get $b_2+\bar b_2=b_1\bar b_1$
   so $\operatorname{Real}(b_2)=\frac{\lambda_1^2}{2}$. We see $\operatorname{Im}(b_2)$ is  an element of~$k$, independent of $\lambda_1$. Setting $\lambda_2=\operatorname{Im}(b_2)$, we can write $b_2=f_1(\lambda_1)+\lambda_2\Delta$ with $f_1(\lambda_1)$  a function of $\lambda_1$ which is independent from~$\lambda_2$.

In the next step, we obtain $b_3-\bar b_3=b_2\bar b_1-b_1\bar b_2,$ from which $\operatorname{Im}(b_3)=2\lambda_1\lambda_2$. We observe $\operatorname{Real}(b_3)$ is an element of $k$ independent from $\lambda_1
$ and $\lambda_2$. We write this time $b_3=\lambda_3+f_2(\lambda_1,\lambda_2)\Delta$, for a function $f_2(\lambda_1,\lambda_2)$  independent from $\lambda_3$.

In general, when $i$ is even, $b_i+\bar b_i$ is a $k$-valued function $f_i(\lambda_1,\dots,\lambda_{i-1})$,  and  $\operatorname{Im}(b_i)=f_{i-1}(\lambda_1,\dots, \lambda_{i-1})$ is a new variable, getting $b_i=\lambda_i$.
When~$i$ is odd,  $b_i-\bar b_i$   determines a $k$-valued function, we set this time  $\operatorname{Im}(b_i)=f_{i-1}(\lambda_1,\dots,\lambda_{i-1})$ and $\operatorname{Real}(b_i)=f_{i-1}(\lambda_1,\dots, \lambda_{i-1}$ for a new variable independent from   $ \lambda_1,\dots,\lambda_{i-1}$. We can write $b_i=\lambda_i+  f_{i-1}(\lambda_1,\dots,\lambda_{i-1})\Delta$.      The result now follows.
\end{proof}

\begin{remark}
We observe that in fact the functions $f_i$ have the property
\begin{gather*}
f_i(0,0,\dots,0)=0.
\end{gather*}
\end{remark}

  The f\/ield $k=\mathbb F_q$ (where $q=p^t$, $p$ an odd prime number) is a $t$-dimensional vector space over~$\mathbb F_p$. We set $e_1,\dots,e_t$ for a basis of~$k$ over~$\mathbb F_p$.

  We describe the elements of $U_0$ as in Proposition~\ref{uo}, and we def\/ine:

\begin{definition}
Let $i$ be relatively prime to $p$, and $l=1,\dots,t$. $H_{i,l}$ denotes the cyclic  subgroup of $U_0$ of order $d=d_{i,l}$ ($d$ is a~power of~$p$), generated by
$z=1+e_lx^i+\alpha_2x^{2i}+\cdots+\alpha_{\operatorname{ord}(z)}x^{\operatorname{ord}(z)i}$   for $i$ odd, and generated by
 $z=1+e_l\Delta x^i+\alpha_2\Delta^2 x^{2i}+\cdots +\alpha_{\operatorname{ord}(z)}\Delta^{\operatorname{ord}(z)}x^{\operatorname{ord}(z)i}$ for $i$ even, where $\operatorname{ord}(z)$ is the integer such that $\operatorname{ord}(z)i<n,$ but $(\operatorname{ord}(z)+1)i\geq n$, and
 $\alpha_j$ are certain elements in~$k$.
\end{definition}

 We determine now the elements~$\alpha_i$ for~$i$ odd (the other case being similar).
 The condition
\begin{gather*}
\big(1+e_lx^i+\alpha_2x^{2i}+\cdots\big)\big(1+e_lx^i+\alpha_2x^{2i}+\cdots\big)^*=1
\end{gather*}
leads to the system (with $\alpha_1=e_l$)
\begin{gather*}
1=1,\\
2\alpha_2-\alpha_1^2= 0,\\
0=0,\\
2\alpha_4-2\alpha_3   +\alpha_2^2=0,\\
\cdots\cdots\cdots\cdots \cdots\cdots\\
\sum_{s=1}^{j} (-1)^{s-j}\alpha_s\alpha_{s-j}=0,
\end{gather*}
 which has the solution $\alpha_2=\frac{1}{2}\alpha_1^2$,
$\alpha_j=0$ for~$j$ odd and for $j$ even greater than~$2$. So, $\alpha_j$ is computed inductively by $2\alpha_j +\sum\limits_t^{\frac{j-2t}{2}}{ \alpha_{2t}\alpha_{j-2t}}=0$.

 \begin{proposition} \label{intersection}
The intersection of    $\prod\limits_{v,u} H_{v,u} $ with the subgroup $H_{i_0,l_0}$ is~$1$  $[1\leq v\leq n-1$, $v$~relatively prime  to~$p$, $v> i_0$  if   $u=l_0$,  and   $v\geq i_0$  if   $u\neq l_0$, $u=1,\dots,t]$.
\end{proposition}

\begin{proof}
Let $H_{i_0,l_0}$ with $i_0$  relatively prime to~$p$.  We prove the case when~$i_0$ is odd (the case when $i_0$ is even is similar).

Let $z\in H_{i_0,l_0}$.  Then
$z= (1+e_{l_0}x^{i_0}+\alpha_2 x^{2i_0} +\cdots+\alpha_{\operatorname{ord}(z)}x^{\operatorname{ord}(z)i_0})^j$, where  $1\leq {j}\leq d$ ($d$~is a~power of $p$) is such that $i_0d>n$.
 We have $j=p^a s$, where $a>0$, $\operatorname{g.c.d.}(s,p)=1$ if $p$ divides $j$, and $a=0$ if $j=s$ is relatively prime to~$p$. Then, we can write
 $z=1+s_0e_{l_0}x^{p^a i_0}+ \text{terms of higher degree}$ ($s\equiv s_0$ (${\rm mod}\, p)$, $1\leq s_0<p$).

If $z$ is also an element of the product above, then
\begin{gather*}
z=\big(1+\beta_1e_{l_1} x^{p^{c_1}i_{l_1}}+\cdots\big)\big(1+\beta_2e_{l_2} x^{p^{c_2}i_{l_2}}+\cdots\big)\cdots ,
\end{gather*}
($1\leq\beta_r<p$, $r=1,\dots,m$) for some $m$, and $i_r>i_0$ if $l_r=l_0$, $i_r\geq i_0$ if $l_r\neq l_0$.
So, the lower degree term of $z$,  as an element of the product must be of the form $(\beta_1e_{l_1} +\beta_2e_{l_2}+\cdots +\beta_me_{l_m})x^{p^{a}i_0}$  with $   p^{c_1}i_1=\cdots =p^{c_m}i_m=p^{a}i_0$. But then  $\beta_1e_{u_1} +\beta_2e_{u_2}+\cdots +\beta_me_{u_m}=s_0e_{l_0}$.

We have two possible cases according to whether $l_0$ appears in the factors for~$z$ (as an element of a product  of~$H$'s) or not.
In the f\/irst case,  since the~$e$'s are linearly independent, we must have (reordering if necessary) $\beta_1=s_0$, $\beta_2=\cdots =\beta_m=0$ and $l_1=l_0$. The equality $p^{c_1}i_1=p^{a}i_0$ is contradictory because $i_1>i_0$, and $i_0$ and~$i_1$  are relatively prime to~$p$.
In the second case, we would have linear dependency between the~$e$'s.

From here, the result follows.
\end{proof}

The next lemmas are direct result of the def\/inition of the integer part function.

\begin{lemma}
Let $a$, $b$ be positive integers such that $a<b$, then the number of multiples of $p$ in the interval $]a,b]$ is
$\lfloor \frac{b}{p}\rfloor-\lfloor \frac{a}{p}\rfloor$.
\end{lemma}

\begin{lemma}
Let $n$, $p$ be positive integers with $p$ a prime element. Then
\begin{gather*}
\left\lfloor \dfrac{\big\lfloor\frac{n}{p^j}\big\rfloor}{p}\right\rfloor=\left\lfloor\dfrac{n}{p^{j+1}}\right\rfloor.
\end{gather*}
\end{lemma}

We assume $n$ is relatively prime to $p$. Consider the maximal integer $r$ such that $n=p^{r-1}m+R$, with $0<R<p^{r-1}$.

\begin{proposition} \label{prop10}
The cardinality of the subgroup $H_{i,u}$,  for $u=1,\dots,t$, where $i$ belongs to the interval $\big]\big\lfloor \frac{n}{p^{j}}  \big\rfloor,\big\lfloor \frac{n}{p^{j-1}}  \big\rfloor\big]$ and~$i$ is relatively prime to~$p$, is~$p^j$.
\end{proposition}

\begin{proof}
We observe that to f\/ind the order of  $H_{i,u}$ we need to f\/ind the  integer $\alpha_i$ such that
$ip^{\alpha_{i}-1}<n< ip^{\alpha_i}$, from which $H_{i,u}$ will have order $p^{\alpha_i}$.
Let us write $n=a_{r-1}p^{r-1}+a_{r-2}p^{r-2}+\cdots +a_1p+a_0$ with $0\leq a_i<p$ and $a_{r-1}\neq 0$.
Then
\begin{gather*}
\left\lfloor \frac{n}{p^j}\right\rfloor =a_{r-1}p^{r-1-j}+a_{r-2}p^{r-2-j}+\cdots +a_j,\\
\left\lfloor \frac{n}{p^{j-1}}\right\rfloor =a_{r-1}p^{r-j}+a_{r-2}p^{r-1-j}+\cdots +a_jp+a_{j-1}.
\end{gather*}
The conditions on $i$ and $n$ say
\begin{gather*}
a_{r-1}p^{r-1-j}+a_{r-2}p^{r-2-j}+\cdots +a_j< i\leq a_{r-1}p^{r-j}+a_{r-2}p^{r-1-j}+\cdots +a_jp+a_{j-1},
\end{gather*}
from which we have the inequalities
\begin{gather*}
p^{j-1}i\leq a_{r-1}p^{r-1}+a_{r-2}p^{r-2}+\cdots +a_jp^j+a_{j-1}p^{j-1}\leq n,
\end{gather*}
but since $n$ is relatively prime to $p$, we have $p^{j-1}i<n$,
and
\begin{gather*}
a_{r-1}p^{r-1}+a_{r-2}p^{r-2}+\cdots +a_jp^j< p^ji.
\end{gather*}
Using the base $p$ expansion of $n$, we get from this last inequality that $n< p^ji$. Taking $\alpha_i=j$ the result follows.
\end{proof}

\begin{corollary}\label{u-cero}
The group $U_0$ is the direct product of all subgroups $H_{i,u}$ as above, where $u$ varies from~$1$ to~$t$, and~$i$ is less or equal to~$n-1$ and relatively prime to~$p$.
\end{corollary}

\begin{proof}
According to Propositions~\ref{intersection},~\ref{prop10}  and Lemma~\ref{order} it is enough  to show that the order of the direct product  of all subgroups~$H_{i,u}$  is~$q^{n-1}$.

Let us f\/ix an element $u$ ($u=1,\dots,t)$. First, notice that there are no multiples of $p$ in the interval $\big\lfloor1, \big\lfloor \frac{n}{p^{r-1}}\big\rfloor\big\rfloor$.  In general, the number of multiples of $p$ in  $\big[\big\lfloor \frac{n}{p^{r-j}}\big\rfloor+1 , \big\lfloor \frac{n}{p^{r-j-1}}\big\rfloor\big]$ is
\begin{gather*}
\left\lfloor\frac{\big\lfloor\frac{n}{p^{r-j-1}}\big\rfloor}{p}\right\rfloor-\left\lfloor\frac{\big\lfloor \frac{n}{p^{r-j}}\big\rfloor}{p}\right\rfloor=
\left\lfloor\frac{n}{p^{r-j}}\right\rfloor-
\left\lfloor \frac{n}{p^{r-j+1}}\right\rfloor.
\end{gather*}
For each  integer $j\geq 0$, we denote by $\gamma_j$  the number of integers  relatively prime to $p$ in the  interval $\big[\big\lfloor \frac{n}{p^{r-j}}\big\rfloor+1 , \big\lfloor \frac{n}{p^{r-j-1}}\big\rfloor\big]$.  We observe that
\begin{gather*}
\gamma_j=\left( \left\lfloor \frac{n}{p^{r-j-1}}\right\rfloor -\left\lfloor \frac{n}{p^{r-j}}\right\rfloor  \right)-
\left(\left\lfloor\frac{n}{p^{r-j}}\right\rfloor-
\left\lfloor \frac{n}{p^{r-j+1}}\right\rfloor\right).
\end{gather*}
Then,  the order of the subgroup $H_{i,u}$ with $i$ in the interval $\big[\big\lfloor \frac{n}{p^{r-j}}\big\rfloor+1 , \big\lfloor \frac{n}{p^{r-j-1}}\big\rfloor\big]$ is $p^{r-j}$. Hence the product  of the  orders  of all~$H_{i,u}$, for a f\/ixed $u$, is $p^{\sum\limits_{j=0}^{r-1} \gamma_j (r-j)}$.  Now, we verify  that 
\begin{gather*}
\sum_{j=0}^{r-1} \gamma_j (r-j)= r\left\lfloor \frac{n}{p^{r-1}}\right\rfloor
 +(r-1)\left[\left(\left\lfloor \frac{n}{p^{r-2}}\right\rfloor -\left\lfloor \frac{n}{p^{r-1}}\right\rfloor \right)-\left\lfloor \frac{n}{p^{r-1}}\right\rfloor \right]\\
\hphantom{\sum_{j=0}^{r-1} \gamma_j (r-j)=}{}
+(r-2)\left[\left(\left\lfloor \frac{n}{p^{r-3}}\right\rfloor -\left\lfloor \frac{n}{p^{r-2}}\right\rfloor \right)-\left(\left\lfloor \frac{n}{p^{r-2}}\right\rfloor  -\left\lfloor \frac{n}{p^{r-1}}\right\rfloor\right)\right]\\
\hphantom{\sum_{j=0}^{r-1} \gamma_j (r-j)=}{}
+(r-3)\left[\left(\left\lfloor \frac{n}{p^{r-4}}\right\rfloor -\left\lfloor \frac{n}{p^{r-3}}\right\rfloor \right)-\left(\left\lfloor \frac{n}{p^{r-3}}\right\rfloor  -\left\lfloor \frac{n}{p^{r-2}}\right\rfloor\right)\right]\\
\hphantom{\sum_{j=0}^{r-1} \gamma_j (r-j)=}{}
+(r-4)\left[\left(\left\lfloor \frac{n}{p^{r-5}}\right\rfloor -\left\lfloor \frac{n}{p^{r-4}}\right\rfloor \right)-\left(\left\lfloor \frac{n}{p^{r-4}}\right\rfloor  -\left\lfloor \frac{n}{p^{r-3}}\right\rfloor\right)\right] +\cdots\\
\hphantom{\sum_{j=0}^{r-1} \gamma_j (r-j)=}{}
+3\left[\left(\left\lfloor \frac{n}{p^2}\right\rfloor -\left\lfloor \frac{n}{p^3}\right\rfloor \right)-\left(\left\lfloor \frac{n}{p^3}\right\rfloor  -\left\lfloor \frac{n}{p^4}\right\rfloor\right)\right]\\
\hphantom{\sum_{j=0}^{r-1} \gamma_j (r-j)=}{}
+2\left[\left(\left\lfloor \frac{n}{p}\right\rfloor -\left\lfloor \frac{n}{p^2}\right\rfloor \right)-\left(\left\lfloor \frac{n}{p^2}\right\rfloor  -\left\lfloor \frac{n}{p^3}\right\rfloor\right)\right]\\
\hphantom{\sum_{j=0}^{r-1} \gamma_j (r-j)=}{}
+\left[\left(\left(n-1\right) -\left\lfloor \frac{n}{p}\right\rfloor \right)-\left(\left\lfloor \frac{n}{p}\right\rfloor  -\left\lfloor \frac{n}{p^2}\right\rfloor\right)\right] =n-1.
\end{gather*}
From this, it is clear that the product of all $H_{i,u}$' (running also over $u$) is
$({p^{n-1}})^t=q^{n-1}$.
From here the result follows.
\end{proof}

We have found a decomposition of $U$ as a product of cyclic subgroups:
\begin{theorem}
The group $U$ is the direct product of a cyclic subgroup of order $q+1$ and the cyclic subgroups determined by the groups $H_{i,l}$ of Corollary~{\rm \ref{u-cero}}.
\end{theorem}

\subsection*{Acknowledgements}

We thank Pierre Cartier for his comments and suggestions at the beginning of this work.
Both authors were partially supported  by FONDECYT  grant 1120578. Moreover,  the f\/irst author was partially supported by the Universidad Austral de Chile, while the second author was also partially supported by Pontif\/icia Universidad Cat\'olica de Valpara\'iso.

\pdfbookmark[1]{References}{ref}
\LastPageEnding

\end{document}